\documentclass{amsart}
\usepackage{amssymb}
\usepackage[matrix, arrow, curve]{xy}
\usepackage{enumerate}
\usepackage{enumitem}
\newtheorem{theorem}{Theorem}[section]
\newtheorem{lemma}[theorem]{Lemma}
\newtheorem{proposition}[theorem]{Proposition}
\newtheorem{cor}[theorem]{Corollary}  

\newtheorem{definition}[theorem]{Definition}
\usepackage[pdftex]{graphicx,color}
\theoremstyle{definition}

\theoremstyle{remark}
\newtheorem{remark}[theorem]{Remark}

\numberwithin{equation}{section}

\begin{document}

\title
[Laver]
{ Laver tables and combinatorics}

\author{Philippe Biane }
\email{biane@univ-mlv.fr}
\date{}
\address{Institut Gaspard-Monge, universit\'e Paris-Est Marne-la-Vall\'ee,
5 Boulevard Descartes, Champs-sur-Marne, 77454, Marne-la-Vall\'ee cedex 2,
France}

\date{\today}

%%%%%%%%%%%%%%%%%%%%%%%%%%%%%%%%%%%%%%%%%%%%%%%%%%%%%%%

\begin{abstract}
The Laver tables are finite combinatorial objects with a simple  elementary definition, which were introduced by R. Laver in \cite{La} from considerations of logic and set theory. 
Although   these objects exhibit some fascinating properties, they seem to have escaped notice from the combinatorics community. My aim is to give a short introduction to this topic, presenting the definition and  main properties and stating a few open problems, which should arouse the interest of combinatorialists.  
\end{abstract}
\maketitle

%%%%%%%%%%%%%%%%%%%%%%%%%%%%%%%%%%%%%%%%%%%%%%%%%%%%%%%

\tableofcontents

%%%%%%%%%%%%%%%%%%%%%%%%%%%%%%%%%%%%%%
\section{Introduction}  In the $90$'s  the logician R. Laver \cite{La}, \cite{La2}, motivated by questions in the  theory of large cardinals, introduced the so-called Laver tables, which are the subject of this paper. In short, a Laver table is a finite set endowed with a binary operation $\star$, which is left distributive, meaning that 
$p\star(q\star r)=(p\star q)\star(p\star r)$ holds for all $p,q,r$, and in which one of the elements acts, by right multiplication, as a cyclic permutation of the elements of the set. The formal definition is given in section 2 below. Laver established the existence  and uniqueness of such a structure for sets of cardinal $2^n$, for some integer $n$, moreover he  discovered many of their interesting combinatorial properties. Some further work was done at that time, mainly by logicians and algebraists, notably Dehornoy, Dougherty, Dr\'apal, Jech, however it seems that this subject has been largely ignored by combinatorialists  despite the fact that the Laver tables have a strong and highly nontrivial combinatorial content.
In this paper I will present some of the most basic properties of Laver tables, from a combinatorial perspective, but I will also try to give some indications on the notions of set theory which lead to their discovery.
I hope that this paper might encourage combinatorialists to look deeper into this beautiful subject.
As I mentioned above, the Laver tables carry a left distributive operation. The study of such operations is not part of the mainstream of algebraic combinatorics rather,  it comes from two other sources, set theory and the theory of braids.
A lot of  information about these subjects may be found in P. Dehornoy's book \cite{Deh} or his recent preprint \cite{Deh2}. The difficulty in studying left distributive structures comes from the fact that the left distributive identity $p\star(q\star r)=(p\star q)\star(p\star r)$
does not have the same number of terms on each side. This implies that computing with   operations satisfying this identity leads to deep recursions. On the other hand experimental study of the Laver tables, which can be made, up to rather large size, using computers, shows that these tables seem both to satisfy many regularity properties and yet to escape any global description. In particular some basic questions have surprising answers: Laver has proved, using a large cardinal axiom, which is not provable in usual ZFC theory, that the projective limit of the Laver tables is a free system. As we shall see below this statement has a very concrete translation into properties of the Laver tables: it asserts that a certain sequence, with an elementary combinatorial definition, is unbounded. However,  up to now no proof is known of this fact  which does not use this  large cardinal axiom. It is a challenge for combinatorialists to find an elementary proof of this fact (or to disprove it...).
This situation bears some resemblance with   the study of the iteration of rational maps on the complex plane in complex dynamics, which gives rise to Julia sets or the Mandelbrot set (see e.g. \cite{Milnor}): these objects also are constructed by  very simple recursive laws, they exhibit some regularity and beautiful features, which  can be seen on the computer generated  pictures that are easily found on the internet, yet their structure is very complicated and many questions concerning them are still open. I hope that Laver tables might attract similar attention from the combinatorialists which would lead to much progress.

This paper is organized as follows: in the next section we give the definition of Laver tables and  make some general comments on left-distributive operations. In section \ref{basic} we give some basic properties of Laver tables then, in section \ref{comefrom}, we explain how these objects were discovered by  Laver, starting from considerations of set theory. We go on with some further properties of Laver tables in sections \ref{further} and \ref{rows}.
 We explain in section \ref{interest} why these are interesting and subtle combinatorial objects. In particular we state some difficult open problems. In section \ref{freq}, we show that the periods of Laver tables have asymptotic frequencies, which define a probability measure on ${\bf N}\cup\{\infty\}$. It would be interesting to describe more precisely this probability distribution. Finally, in section \ref{Max}, we introduce a particular class of elements of the Laver tables. These are, in some sense, the simplest elements and a remarkable property of these elements is that they form a  subset which is stable under the operation  of the Laver table. Moreover they are parameterized by  binary partitions which are objects related to more  mainstream algebraic combinatorics. 
Most of the results of this paper are not new, except perhaps the content of section \ref{freq} and \ref{Max}, and can be found in the papers cited in the bibliography, although sometimes in slightly different form, so that I did not try to track down the exact reference for each of them.

I would like to thank Patrick Dehornoy for introducing me to this beautiful subject, as well as Ales Dr\'apal for  communicating me his preprint \cite{Dra4}. Both of them made useful remarks on a first version of this paper.
\section{What are Laver tables?}\label{laverdef}
\subsection{A binary operation}
Let $N$ be a positive integer, there exists a unique binary operation $\star$ on the set $\{1,2,\ldots,N\}$ such that, for all $p,q$
\begin{align}\label{ld1}p\star 1&=p+1 \mod (N)\\
\label{ld}p\star(q\star 1)&=(p\star q)\star(p\star 1)
\end{align}
Indeed
property (\ref{ld1}) implies $$N\star 1=1$$ then using (\ref{ld}) one gets
$$N\star2=N\star(1\star 1)=(N\star 1)\star(N\star 1)=1\star 1=2.$$ By induction on $q$ one has:
$$N\star q=N\star((q-1)\star 1)=(N\star (q-1))\star(N\star 1)=(q-1)\star 1=q.$$
Starting from $(N-1)\star 1=N$ and  the relation

\begin{equation}\label{induction}
p\star(q+1)=p\star(q\star 1)=(p\star q)\star (p+1)
\end{equation} we can  use induction, descending on $p$ and ascending on $q$,  
to prove that $p\star q$ is well defined and satisfies  $N\geq p\star q>p$. 
\subsection{The Laver tables}
It turns out that the binary operation  $\star$, defined  above, is left distributive if and only if   
 $N=2^n$ for some $n\geq 0$. 
Left distributivity is the property that, for all $p,q,r$ one has
\begin{equation}
\label{ld0}p\star(q\star r)=(p\star q)\star(p\star r).
\end{equation}
Note that (\ref{ld0}) is 
  (\ref{ld})  with $1$ replaced by any $r\in [1,N]$.
The proof of this result is elementary, but non trivial, and can be found for example in the books \cite{Deh}, \cite{Deh1} or in the survey by A. Dr\'apal \cite{Dra3}. In the sequel  I will denote by $\star_n$ the operation on $[1,2^n]$ thus obtained.
Here  is the Laver table giving the values of $p\star_n q$, for $N=4$:
\medskip

\renewcommand{\arraystretch}{1.5}
\begin{center}
\begin{tabular}{c|cccc}
 $\star$&\bf 1 &\bf  2&\bf 3&\bf 4\\
\hline
{\bf 1}&  2&4&2&4
\\
{\bf 2}& 3&4&3&4
\\
{\bf 3}& 4 &4&4&4
\\
{\bf 4}&1 &2 &3& 4
\end{tabular}
\end{center}
and for $N=8$:

\begin{center}
\begin{tabular}{c|cccccccccc}
 $\star$&\bf 1 &\bf  2&\bf 3&\bf 4&\bf 5&\bf 6&\bf 7&\bf 8\\
\hline
{\bf 1}&  2&4&6&8&2&4&6&8
\\
{\bf 2}& 3&4&7&8&3&4&7&8
\\
{\bf 3}& 4 &8&4&8&4&8&4&8
\\
{\bf 4}& 5& 6 &7& 8&5&6&7&8
\\
{\bf 5}& 6& 8&6&8&6&8&6&8
\\
{\bf 6}& 7& 8& 7& 8&7&8&7&8
\\
{\bf 7}&8 &8 &8& 8&8&8&8&8
\\
{\bf 8}&1 &2 &3& 4&5&6&7&8
\end{tabular}
\end{center}
The $p$'s are in the first column and the $q$'s in the first row.
\subsection{Left distributive operations}\label{LD1}
Before going further into a combinatorial exploration  of the Laver tables, 
I will make some general remarks on left distributive operations. A thorough study of these, with many examples and applications, notably to knots and braids, can be found in the monograph by P. Dehornoy \cite{Deh}.
A binary operation $\star:S\times S\to S$ on a set $S$ is said to be {\sl left distributive} if it satisfies
(\ref{ld0}) for all $p,q,r\in S$.
 A good way to think about  property (\ref{ld0}) is to notice that the left multiplication operator by some element $s$, denoted by $\lambda_s$ i.e. $\lambda_s(t)=s\star t$, is a homomorphism of the structure: for all $p,q,r$ one has
\begin{equation}
\lambda_p(q\star r)=\lambda_p(q)\star\lambda_p( r)
\end{equation}
 thus $\lambda$ gives a map $S\to Hom(S,\star)$. However, in general, $\lambda_{p\star q}$ does not have an obvious relation to $\lambda_p$ and $\lambda_q$, in particular it is not equal to $\lambda_p\circ \lambda_q$!
A motivation for considering such a property comes from the study of the set of maps of a set to itself $f:X\to X$.
Composition of maps  gives a semigroup structure on this set. Identifying a function with its graph $G_f=\{(x,f(x)),x\in X\}$, the graph of the composition $f\circ g$ is
 $$(Id\times f)(G_g)=\{(x,f(g(x))),x\in X\}.$$ 
It is also possible to ``apply" the function $f$ to the graph $G_g$ to produce the set 
$$(f\times f)(G_g)=\{(f(x),f(g(x)),x\in X\}.$$ 
In general this is not the graph of a  function, unless $f$ is a bijection, in which case it is the graph of the function $f\circ g\circ f^{-1}$. It turns out that the operation $f
\star g=f\circ g\circ f^{-1}$ is left distributive in the following sense: if $f$ and $g$ are invertible then 
$f\star(g\star h)=(f\star g)\star(f\star h)$. Actually, a trivial computation shows that, in any group, the conjugation operation $f\star g=fgf^{-1}$ is left-distributive. It is interesting however, for reasons which will appear later, to enlarge the previous example in the following way:
for a set $X$  let $I_X$ be the set of partially defined injections  $f:D_f\subset X\to X$ where $D_f$ is the domain of definition of $f$, containing in particular ``the injection with empty domain". The set $I_X$ is a semigroup for the natural notion of composition (which may result in an injection with empty domain). Each such injection can be described by its graph $\{(x,f(x)),x\in D_f\}\subset X\times X$. If $g\in I_X$ then $(g\times g)(G_f)$ is the graph of a partially defined injection $g\star f$ and again it is easy to see that the operation $\star$ is left distributive and that $\lambda_f\lambda_g=\lambda_{f\circ g}$ where $f\circ g$ denotes the composition of partially defined injections. Observe that one has 
\begin{equation}\label{function}
f\star g(f(x))=f(g(x))
\end{equation}
 whenever the two members of  this equality are defined. This  serves as a substitute for the formula $f\star g=f\circ g\circ f^{-1}$. 
At this stage, a natural question is whether one can find a family $I$ (not reduced to the identity) of {\sl increasing, everywhere defined} injections  of $\bf N$ into itself, endowed with a binary left distributive operation $\star:I\times I\to I$, such that, for any $\iota,\eta\in I$ and $x\in \bf N$ one has
$\iota\star\eta(\iota(x))=\iota(\eta(x))$, as in (\ref{function}). It turns out that this is a highly nontrivial question to which, 
as we shall see in section
\ref{interest}, the Laver tables give a surprising answer.

\section{Basic properties of the Laver tables}\label{basic}
We now come back to the Laver tables constructed in section \ref{laverdef} and describe some of their elementary properties. 
\subsection{Periods and  projective  limits}
The following properties of the operation $\star_n$ are easily established by induction, see e.g. \cite{Deh} (as above we put $N=2^n$).
\begin{itemize}

\item
 For all $p\in [1,N]$ one has  $N\star_np=p$ and  $p\star_nN=N$ 

\medskip

\item For every $p\in[1,N]$ the sequence 
$p\star_nq,\ q=1,2,\dots$ is periodic, with period  $\pi_n(p)$, a power of  $2$, and 
the sequence
$p\star_nq,\ q=1,2,\dots, \pi(p)$ is strictly increasing from  $p\star_n 1=p+1$ to $p\star_n\pi(p)=N$. 

\medskip

\item\label{modulo}
The projection  $\Pi_{N}:[1,2N]\to[1,N]$ modulo $N$ is a homomorphism.

\medskip
\item
For all $p,q$ one has
\begin{equation}\label{puissances}
p\star_nq=(p+1)^{( q)}
\end{equation}
where the {\sl left  powers\footnote{One can similarly define right powers but we will not use them.}} $x^{(k)}$ are defined by $x^{(1)}=x,\,x^{(k+1)}=x^{(k)}\star_n x$. This follows at once from (\ref{induction}).
\end{itemize}
In order to illustrate these  properties let us display again the Laver table of size $8$. 
In the last column we show the period of each row and
we  divide the table
into four squares. It is then immediate to check that each of these squares  is equal, modulo $4$,  to the table of order $4$.
$$\begin{array}{c|cccc|cccc|cl}
 \star&\bf 1 &\bf  2&\bf 3&\bf 4&\bf 5&\bf 6&\bf 7&\bf 8&\pi\\
\hline
{\bf 1}&  2&4&6&8&2&4&6&8&4
\\
{\bf 2}& 3&4&7&8&3&4&7&8&4
\\
{\bf 3}& 4 &8&4&8&4&8&4&8&2
\\
{\bf 4}& 5& 6 &7& 8&5&6&7&8&4
\\
\cline{1-9}
{\bf 5}& 6& 8&6&8&6&8&6&8&2
\\
{\bf 6}& 7& 8& 7& 8&7&8&7&8&2
\\
{\bf 7}&8 &8 &8& 8&8&8&8&8&1
\\
{\bf 8}&1 &2 &3& 4&5&6&7&8&8
\end{array}
$$
It is possible to take a projective limit of the Laver tables with respect to the natural projections $\Pi_{n,m}:[1,2^n]\to[1,2^m],\ n>m$ and obtain a left distributive operation on the set of 2-adic integers. This left distributive system is generated by $1$.
\subsection{Generators and relations}
Consider the free system with one generator for the left distributive operation (\ref{ld0}) namely, denoting the generator by $1$, it consists of all well parenthesized expressions in $1,\star$, like $1,1\star 1,(1\star 1)\star 1,1\star(1\star 1)$, etc. equipped with the operation $\star$, modulo the congruence induced by the relation (\ref{ld0}). The Laver table of order $2^n$, as  a left distributive system with one generator, satisfies supplementary relations, for example one has 
\begin{equation}\label{relation}
1^{(2^n+1)}=1.
\end{equation}
  In fact one can show that the Laver table of order $2^n$ is exactly the left distributive system with one generator $1$ and the relation (\ref{relation}), see e.g. \cite{Deh}, \cite{Weh}. A very deep question is whether the projective limit of Laver tables is a free system. We will say more about this in section \ref{interest}.
\subsection{Homomorphisms and semigroup structure}
As we have remarked, for any $p$ the left multiplication by $p$ is a homomorphism for the operation $\star_n$. More generally, if $p\in [1,2^n]$ has period $\pi_n(p)\leq 2^m$ then the map $q\mapsto p\star_n q$ is a homomorphism from $[1,2^m]$ to $[1,2^n]$ (with respect to their respective operations $\star_m,\star_n$).  Conversely, for any homomorphism $\varphi:[1,2^m]\to [1,2^n]$ one has, using  (\ref{puissances}):
$$\varphi(q)=\varphi(1^{(q)})=(\varphi(1))^{(q)}=p\star_nq$$
with $p=\varphi(1)-1\mod 2^n$ so that $\varphi$ is given by left multiplication by $p$.
Since  composition of homomorphisms is a homomorphism, for any $p,q\in [1,N]$, there exists a unique $s$, denoted $p\circ_n q$, such that $\lambda_s=\lambda_p\circ \lambda_q$   or, equivalently, 
\begin{equation}\label{compose}
p\star_n(q\star_nr)=s\star_nr\quad \text{for all $r$.}
\end{equation}
  Using equation (\ref{compose}) for $r=1$ we see that 
$p\circ_n q$ is characterized by the relation
\begin{equation}\label{composeind}
(p\circ_n q)+1=p\star_n(q+1)\mod N
\end{equation} 
which relates $\star_n$ and $\circ_n$.
The product $\circ_n$ is  associative and gives a semigroup operation on $[1,N]$, in fact the map $\lambda$ gives an isomorphism 
 $([1,N],\circ_n)\sim \text{End}([1,N],\star_n)$. One can check that $\star_n$ and $\circ_n$  satisfy the properties:

\begin{align}
p\star_n(q\circ_n r)&=(p\star_nq)\circ_n(p\star_n r)\label{01}\\
(p\circ_n q)\star_n r&=p\star_n(q\star_n r)\\
(p\star_n q)\circ_n p&=p\circ_n q.\label{03}
\end{align}

These relations might seem less strange if one observes that they are also satisfied by any pair $\circ,\star$ where $\circ$ is a group operation and $\star$ the associated conjugation operation, $a\star b=a\circ b\circ a^{-1}$, as in section \ref{LD1}.
\subsection{Backwards notation}
We have seen that the natural projection $$\Pi_N:[1,2N]\to~ [1,N]$$ is a homomorphism.
The embedding: 
$$
\begin{array}{l}
\iota_N:([1,N],\star_n)\to([1,2N],\star_{n+1})\\
 p\mapsto p+N=N\star_{n+1} p
\end{array}
$$
is  also a homomomorphism. This implies that,   
for nonnegative  integers $  p,q,$  the value of
\begin{equation}\label{*}
p* q:=N-(N-p)\star_n(N-q)
\end{equation} 
 does not depend on $N$, as long as $p,q< N=2^n$. One can therefore take an inductive limit with respect to the embeddings $\iota_N$ and build an infinite  table giving the values of  $p* q$ for nonnegative integers $p$ and $q$. 
The set of nonnegative integers is thus endowed with a left distributive operation $*$.
The properties of $\star_n$ immediately translate into the following properties of $*$:
\begin{itemize}

\item  $0*p=p$ and $p*0=0$ for all $p\geq 0$.
\item For every $p>0$ the sequence 
$p* q,\ q=0,1,2\dots$ is periodic, with period  $\pi(p)$, a power of  $2$.
\item
$p* q,\ q=0,1,2,\dots ,\pi(p)-1$ is strictly increasing and  $p*(\pi(p)-1)=p-1$.
\item For all  $p,q,r,n$:  one has\begin{equation}\label{autodist}p*(q*r)=(p*q)*(p*r)\end{equation}
\begin{equation}\label{modulo2}(p\mod 2^n)*(q\mod 2^n)=p*q\mod 2^n\end{equation}
\item The formula $p\circ q=p*(q-1)+1$ defines a semigroup structure on the set of positive integers.
\item One has 
\begin{equation}\label{puissancesback}
p*(2^n-q)=(p-1)^{(q)}\qquad\text{for $2^n\geq\pi(p)$ }\end{equation}
(with a suitable definition of the left powers).
\end{itemize}

Equations (\ref{01})$\cdots$(\ref{03}) also hold for $*$ and $\circ$.
 It follows in particular that for every $m$ the interval 
$[0,m]$ is closed under $*$ and $[1,m]$ is closed under $\circ$. Since the operations $\star_n$ or $*$ are related by the map $p\mapsto N-p$ they are equivalent but, depending on the aspects of Laver tables  one wants to consider, often one of them turns out to be more convenient than the other. 
\subsection{Computation of the Laver table}
We saw in section \ref{laverdef} how to  compute  the products $p\star_n q$ by induction. For the convenience of the reader 
I will illustrate here the computation of the operation  $*$, which gives the inductive structure and which, of course, is equivalent to the computation of the operations $\star_n$.
The Laver table recording the  $p*q$ can be contructed by induction on the rows. If one knows  rows from $0$ to $p-1$ the  row of $p$ is obtained as follows:
for  $n$ large enough (i.e. $2^n>p$) one has
\begin{equation}\label{2n-1}p*(2^n-1)=p-1\end{equation}
 then 
\begin{equation}\label{2n-2}p*(2^n-2)=p*((2^n-1)*(2^n-1))=(p*(2^n-1))*(p*(2^n-1))=(p-1)*(p-1)\end{equation}
and, by induction on $k$, using $2^n-k-1=(2^n-k)*(2^n-1)$ and (\ref{ld0}):
\begin{equation}\label{recurrence}
p*(2^n-k-1)=(p*(2^n-k))*(p-1).
\end{equation}
By the periodicity properties of the Laver tables, one has  $p*(2^n-2^m)=0$ for some  $m< n$ so that  the period of  $p$ is $2^m$. Once this value is reached the row is completed by periodicity.  As an example we  compute the row of $7$ assuming rows from $0$ to $6$ have been computed (all rows between $1$ and $6$ have period $\leq 4$):
$$\begin{array}{ccccccccccc}
 &\bf 0 &\bf 1&\bf  2&\bf 3&\bf 4&\bf 5&\bf 6&\bf 7&\bf 8&\\
{\bf 0}&  0&1&2&3&4&5&6&7&8\\
{\bf 1}&  0&0&0&0&0&0&0&0&0
\\
{\bf 2}& 0&1&0&1&0&1&{\boxed 0}&1&0
\\
{\bf 3}& 0 &2&0&2&0&2&0&2&0
\\
{\bf 4}& 0& 1 &2& 3&0&1&{\boxed  2}&3&0
\\
{\bf 5}& 0& 4&0&4&0&4&0&4&0
\\
{\bf 6}& 0& 1& 4& 5&0&1&{\boxed  4}&5&0
\\
{\bf 7}& 0&2 &4 &6 &{\boxed  0}&{\boxed 2}&{\boxed 4}&{\boxed 6}&0
\end{array}
$$
In order to get row 7  one takes  $n=3$, then $2^n-1=7$ and by (\ref{2n-1}) one has $7*7=6$. Applying  (\ref{2n-2}) gives 
$7*6=6*6=4$, then by repeatedly applying (\ref{recurrence}) we get $7*5=4*6=2$, $7*4=2*6=0$. The rest of the row follows by periodicity. The relevant values are shown in boxes in the above table.

It is easy to make a computer program which performs these computations for larger values of $p$ and $q$, however one encounters quickly memory size problems. We will see in section \ref{seuils} how to encode the Laver tables in a more compact form. A  formula expressing $p*q$ is known for some classes of $p$'s (for example if $p$ is a power of $2$, see section \ref{rows} below) but no formula is known  in the general case and it is likely that  no such formula exists.

\section{Where do Laver tables come from?}\label{comefrom}
The Laver tables are finite combinatorial objects with a very simple and elementary definition, however they were discovered in the context of the theory of large cardinals, a part of mathematics which seems quite far from finite combinatorics. Although logic and set theory are not in my domain of expertise, I will try to convey some idea of the set theoretical objects involved in this construction, without giving complete definitions, and refer to the books \cite{Deh}, \cite{Deh1} for a thorough discussion. I assume here only a very basic knowledge of ordinals.

  Recall that a set $X$ is infinite if there exists an injection $j:X\to X$ which is not surjective. The typical example is the map $j:x\to x+1$ on the set $\bf N$ of natural numbers. We will need also the notion of  an elementary embedding of a structure $X$ into itself, which  is an injective map $j:X\to X$ such that any formula in the language of $X$ is true if and only if its image by $j$ is true.
We consider now ordinals. Recall that ordinals are totally ordered $$0<1<2<\ldots<\omega<\omega+1<\ldots.$$ Any ordinal $\lambda$ has a successor $\lambda+1$, but some ordinals, like $\omega$ the first infinite ordinal,   are not the successor of any ordinal.
They are called {\sl limit ordinals}. To ordinals one can associate {\sl ranks} $V_\lambda$ which  are sets  defined by induction starting with $V_0=\emptyset$ and satisfying  $V_{\lambda+1}=2^{V_\lambda}$ and $V_\lambda=\cup_{\mu<\lambda}V_\mu$ if $\lambda$ is a limit ordinal.
The rank $V_\lambda$ is equipped
with the language of first order set theoretic formulas. Laver postulated the existence of a limit ordinal $\lambda$ and a nontrivial elementary embedding $j$ of $V_\lambda$ into itself. Consideration of such an object is a natural extension of the idea of an infinite set: a set is infinite if it  is ``isomorphic" to one of its subsets without being equal to it, except that here the notion of ``isomorphic" is taken in a very strict sense, that of ``satisfying the same first order properties" which is much stronger  than just ``being in bijection with".

The image by $j$ of any formula 
 which is true in $V_\lambda$ remains true. It follows that all ``ordinary sets" i.e. the ones which can be constructed from the empty set by using von Neumann construction of taking subsets, power sets, unions, etc.,  which form the universe in which almost all ordinary mathematics is done, are invariant under $j$, moreover $j$ is monotonous. Since $j$ is assumed to be nontrivial there must exist a smallest ordinal $\kappa<\lambda$ such that $j(\kappa)>\kappa$, this ordinal is called the {\sl critical ordinal} of $j$.
By elaborating on the remark above one can see that $\kappa$ has to be inaccessible, so that, for example,  for every ordinal $\mu<\kappa$ one has $2^\mu<\kappa$. The ordinals $\kappa$ and $\lambda$ have therefore to be very large.
 The existence of such an object cannot be proved in the usual axiomatic system of ZFC,  it has to be introduced by adding a new axiom. Once the existence of $j$ is granted by this new axiom, one can construct new elementary embeddings: an obvious way is to compose $j$ with itself, but the structure of set theory allows also another construction, reminiscent of what we saw in section \ref{LD1}:
given an elementary embedding $l$ of $V_\lambda$, its restriction to some $V_\mu$ with $\mu<\lambda$ is itself a set in $V_\lambda$ (it can be identified to its graph as a function) and one can apply another elementary embedding $k$ to this set. Taking inductive limits over $\mu<\lambda$ one gets an elementary embedding $k\star l$. This gives a new operation on elementary embeddings. One can prove that this operation is left distributive (essentially for the reason explained in section \ref{LD1}).
Take now $J$ to be the set of all elementary embeddings obtained from $j$ by using $\star$ (e.g. $j,j\star j,(j\star j)\star j, etc.$) then every such elementary embedding is nontrivial and has a critical ordinal. One can prove that these ordinals form an increasing sequence $\kappa_0<\kappa_1<\kappa_2<\ldots$, moreover one can define a notion of equivalence modulo $\kappa_n$ on elementary embeddings so that the Laver table of order $2^n$ is obtained from $J$ by taking the quotient with respect to this equivalence relation, with $j$ corresponding to $1$ and $\star$ giving $\star_n$. The composition of elementary embeddings gives another operation, which yields $\circ_n$ by passing to the quotient. The details are somewhat technical and can be found in \cite{Deh}.
Thus we see that Laver tables, which are  finite combinatorial objects, have been discovered from quite elaborate considerations, involving logic and set theory of large cardinals! 

\section{Some further properties of Laver table.}\label{further}
We now describe some more involved properties of Laver tables. Other  results of this kind can be found in the works of Dr\'apal and Dougherty mentioned in the bibliography.
Here we use the operation $*$ for which many of the properties  are more easily stated but of course equivalent statement can be obtained for the operations $\star_n$.
We will use the following notation:
if  $p$ is a positive integer we denote   $\nu_i(p)$ (with $\nu_1<\nu_2<\ldots$) the powers of 2 arising in the binary expansion  of   $p$, i.e. $p=\sum_{i=1}^r2^{\nu_i(p)}$.

\subsection{Adding a power of $2$}

For all $p,q,n$ with $0<p<2^n$ it follows from (\ref{modulo2}) that $(p+2^n)*q=p*q\ \text{or}\ p*q+2^n$ 
\begin{proposition}\label{power1}Let $p$ be such that  $0<p<2^m<2^n$ and $q$ a nonnegative integer, then  
\begin{equation}\label{nm}(p+2^m)*q=p*q+2^m\ \text{if and only if}\ 
(p+2^n)*q=p*q+2^n.\end{equation}
\end{proposition}
\begin{proof}
It is easy to prove the statement for $p=1$, indeed one has $(1+2^n)*q=0$ if $q$ is even and $=2^n$ if $q$ is odd, see section \ref{2m2n} below. Let now
   $\bar p=p+2^m$ with $0\leq p<2^m$ then $\pi(\bar p)\leq 2^m$ therefore for  $l$ large enough one has, using (\ref{recurrence}),
$$\bar p*(2^l-k)=(\bar p*(2^l-k+1))*( p-1)$$ and similarly for $\tilde p=p+2^n$:
$$\tilde p*(2^l-k)=(\tilde  p*(2^l-k+1))*( p-1)$$  The relation is then deduced by a double induction on  $p$ and $k$, noting that one has always $u*v<u$. 
\end{proof}

\medskip
In particular one can define, for every $p$, its ``coperiod" by $\bar \pi(p)=\pi(p+2^n)$, which does not depend on $n$ if $n$ is large enough. One has $\bar\pi(p)/\pi(p)=1\ \text{or}\ 2$.
One can go further and note that, for $p<2^m<2^n$, one has $$(p+2^m+2^n)*q=p*q+i2^m+j2^n$$
 for some  $i,j\in\{0,1\}$.
By an argument  analogous to the one in the above proof one gets easily
\begin{proposition}\label{power2}Let $p,q,m,n,s,t$ be integers such that $p<2^{m-1}$, $m<n$ and $p<2^{s-1}$, $s<t$ 
then for  $i,j\in\{0,1\}$ one has:
 \begin{equation}\label{nmst}(p+2^m+2^n)*q=p*q+i2^m+j2^n\ \text{if and only if}\ 
(p+2^s+2^t)*q=p*q+i2^s+j2^t.\end{equation}
\end{proposition}

Note that the conditions  $p<2^{m-1}$ and $p<2^{s-1}$ are necessary. In general one can prove, using similar arguments: 
\begin{proposition}\label{power3}
If $1\leq p<2^m<2^{n-1}$ then 
\begin{equation}
\pi(p+2^{m+1}+2^n)\leq \pi(p+2^m+2^n)\leq 2\pi(p+2^{m+1}+2^n).
\end{equation}
\end{proposition}
However one may have $\pi(p+2^m+2^n)=2\pi(p+2^{m+1}+2^n)$, for  example taking  $p=5,m=3,n=5$ one has
$\pi(45)=8,\ \pi(53)=4$.
\subsection{The threshold}\label{seuils}
Let $p$ be a nonnegative integer.
Consider the increasing sequence  $p*1,p*2,\ldots ,p*(\pi(p)-1)$. If  $p\leq2^n$ and $q=p+2^n$ then there are two possibilities for 
 $q*1,\ldots,q*(\pi(q)-1)$:

\begin{itemize}\item
either $\pi(q)=2\pi(p)$ and
$$\begin{aligned}
q*r&=p*r\quad \text{for}\ r<\pi(p)\\
q*r&=p*(r-\pi(p))+2^n\quad \text{for}\ \pi(p)\leq r< 2\pi(p)
\end{aligned}$$
\item or  $\pi(q)=\pi(p)$ and there exists some integer $\theta(q)\geq 1$ such that
$$\begin{aligned}
q*r&=p*r\quad \text{for}\ r<\pi(p)-\theta(p)
\\q*r&=p*r+2^n\quad \text{for}\ \pi(p)-\theta(p)\leq r< \pi(p)
\end{aligned}$$ 

\end{itemize}
One defines $\theta(q)=\pi(p)$ in the first case.
The integer $\theta(q)$ is called the threshold of $q$. 
In order to build a row of the Laver table $p*q;q=0,1,\ldots$ for some integer $p$, it is enough to know the thresholds of the numbers $p_j=\sum_{i=1}^j2^{\nu_i(p)}, \, j=1,2,\ldots$. In particular, knowing the sequence of numbers $\theta(p);p=2,3,\ldots$ allows to reconstruct the whole Laver table. This is especially useful for doing computer experiments since, instead of storing all products of the Laver table, one can just  store the sequence of thresholds, which saves a lot of memory space.
As an example we illustrate how to compute the row $p*q$ for $p=494=2^1+2^2+2^3+2^5+2^6+2^7+2^8$ (the sequence  $\nu_i(p)$ is $1,2,3,5,6,7,8$). The table below gives the numbers $p_i$ in the first column,  their thresholds in the second column and, in the last column, the row $p_i*q;\ q=0,1,2,\ldots$, up to the period, with the last $\theta(p)$ numbers underlined in bold characters.
$$
\begin{array}{ccl}
p_i&\theta&p_i*q\\
2&1&0,{\bf\underline 1}\\
6&2&0,1,{\bf\underline 4,\bf\underline 5}\\
14&1&0,1,4,\bf\underline {13}\\
46&4&0,1,4,13,\bf\underline {32}, \bf\underline {33},\bf\underline {36}, \bf\underline {45}\\
110&3&0,1,4,13,32,\bf\underline{97},\bf\underline{100},\bf\underline{109}\\
238&3&0,1,4,13,32,\bf\underline{225},\bf\underline{228},\bf\underline{237}\\
494&8&0,1,4,13,32,225,228,237,\bf\underline{ 256},\bf\underline{257},\bf\underline{260},
\bf\underline{269},\bf\underline{288},\bf\underline{481},\bf\underline{484},\bf\underline{493}
\end{array}
$$

\begin{proposition}
If $p<2^m$ and $m<n$ then $\theta(p+2^m)=\theta(p+2^n)$.
\end{proposition}
 \begin{proof}
This follows  from Proposition \ref{nm}.
\end{proof}
This implies  the existence of a  ``cothreshold" $\bar\theta(p)=\theta(p+2^n)$ with $p<2^n$ which does not depend on  $n$.
Here is the sequence of periods and thresholds for small values of $p$.

\medskip
\renewcommand{\arraystretch}{1.5}
\begin{center}
\begin{tabular}{|c|c|c|c|c|c|c|c|c|c|c|c|c|c|c|c|c|c|c|c|}
\hline
%&&&&&&&&&&&&&&&&&&
%\\
p&1&2&3&4&5&6&7&8&9&10&11&12&13&14&15&16&17&18
%\\
%&&&&&&&&&&&&&&&&&&
\\
\hline
%&&&&&&&&&&&&&&&&&&
%\\
$\pi(p)$&1&2&2&4&2&4&4&8&2&4&4&8&4&4&4&16&2&4
\\
%&&&&&&&&&&&&&&&&&&
%\\
\hline
%&&&&&&&&&&&&&&&&&&
%\\
$\theta(p)$&-&1&1&2&1&2&2&4&1&2&2&4&2&1&1&8&1&2
%\\
%&&&&&&&&&&&&&&&&&&
\\
\hline
\end{tabular}
\end{center}

\medskip

Looking at this table we see that, in this range, $\theta(p)\leq\pi(p)/2$. Actually this is always true, as follows from the following results, due to Dr\'apal \cite{Dra1}, \cite{Dra2}, \cite{Dra3}.
\begin{proposition}\label{idempotents}
For $p\geq 1$ one has  $p*p=0$ if and only if $p$ is a power of $2$.
\end{proposition}
\begin{proof}
It is easy to see that $2^n*2^n=0$ for all $n$.
If $p=2^m+2^n$ with $m<n$ then $p*p=2^m\ne 0$ by Proposition \ref{double}. In general, if $p=\sum_i2^{\nu_i(p)}$ we can reduce modulo 
$2^{\nu_2(p)+1}$, and use the preceding case to prove that $p*p\ne 0$.
\end{proof}
Since $q\mapsto p* q$ is a homomorphism for $*$ it must map elements with square zero to elements with square zero therefore we have
\begin{cor}\label{seuil2}
For any $p$ and $2^k<\pi(p)$ there exists $l\geq k$ with  $p*2^k=2^l$.
\end{cor}
 
Finally we get the following estimate on tresholds.
\begin{proposition}\label{seuil}
Let  $p$ be a positive integer with $p<2^n$ and   $q=p+2^n$, then \begin{itemize}
\item
either $\theta(q)=\pi(p)$ and $\pi(q)=2\pi(p)$ 
\item or $\theta(q)<\pi(p)/2$ and $\pi(q)=\pi(p)$.
\end{itemize}
\end{proposition} 
\begin{proof}
Assume that $\pi(q)=\pi(p)=2^m$ and $\theta(q)>2^{m-1}$ then $q*2^{m-1}=p*2^{m-1}+2^n$ with $0<p*2^{m-1}<2^n$, which is impossible in view of 
Corollary \ref{seuil2}.
\end{proof}
The values of the thresholds in the table above are all powers of $2$. This is not true generally but here 
 is the repartition of pairs $(\theta(p),\pi(p))$ for $p$ between $1$ and $2^{12}$,
(for example there are 761 numbers $1\leq p\leq 2^{12}$ with period  $16$ and threshold $4$):

\begin{center}
\scalebox{0.8}{
$\begin{array}{ccccccccccccccc}
& \bf 2&\bf 4&\bf 8&\bf  16&\bf 32&\bf 64&\bf 128&\bf 256&\bf 512&\bf 1024&\bf2048&\bf 4096\\

\bf 1&12 &2103 & & & & & & & & && \\
\bf 2& &66 & &30 & & & & & &&&\\
\bf 3& & &398 &213 & & & & & &&&\\
\bf 4& & &58 &761 & & & & & &&&\\
\bf 7& & & & 63& & & & & &&&\\
\bf 8& & &&121 & & & & & & &&\\
\bf 16& & & & &12 & 110& & & &&&\\
\bf 32& & & & & & 34& & & &&&\\
\bf 48& & & & & & &19 &6 & &&&\\
\bf 64& & & & & & & 22& 26& &&&\\
\bf 112& & & & & & & &1 & &&&\\
\bf 128& & & & & & & &25 & &&&\\
\bf 256& & & & & & & & &4 &2&&\\
\bf 512& & & & & & & & & &6&&\\
\bf 1024& & & & & & & & & &&2&\\
\bf 2048& & & & & & & & & &&&1\\

\end{array}$}
\end{center}

The data above seem to indicate that   $\theta(p)$ only takes values of the form
$2^i-2^j$.
Indeed the  values of thresholds in this table are, apart from powers of 2:
$$3=2^2-2^0,\ 7=2^3-2^0,\ 15=2^4-2^0,\ 48=2^6-2^4,\ 112=2^7-2^4,$$ (in the table above only values smaller than $2^{12}$ are considered, in fact I pushed the computations to $2^{31}$ and the claim still holds).
 I do not know whether this   property holds for all thresholds.

\subsection{Binary expansion}
Any nonnegative integer can be identified with a finite subset of $\bf N$ by  its binary expansion. The inclusion relation of subsets of $\bf N$  induces an order relation $p\sqsubset q$ on the nonnegative integers, thus $p\sqsubset q$ if and only if the digits of the binary expansion of $p$ are smaller than the corresponding ones of $q$.
The existence of the threshold implies that, for any $p,q$ one has 
\begin{equation}\label{binaire}
p*q\sqsubset p-1.
\end{equation}
Here is a graph of the order relation $\sqsubset$,  for $p\leq 256$, which is essentially a Sierpinsky triangle:
$$\includegraphics[scale=.3]{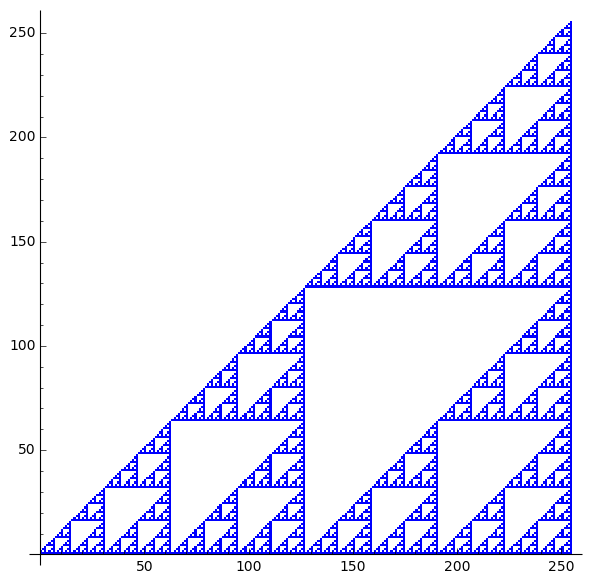}$$
The Laver table can be represented as a subset of this graph, (shifted by 1 to take into account the $p-1$ in 
(\ref{binaire})):  drawing the points $p*q$ above $p$ for each $p$, we get the following picture 

$$\includegraphics[scale=.3]{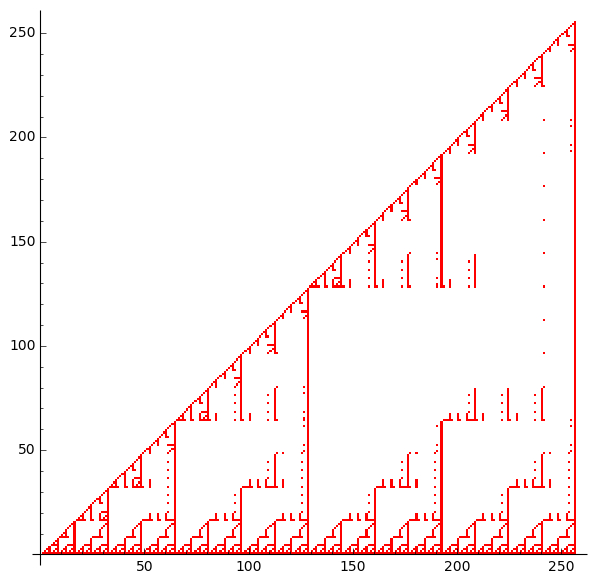}$$

As is clear on this graph, the Laver tables become much sparser than the Sierpinsky triangle as $n$ increases. More on this in sections \ref{freq} and \ref{Max}.
\section{Some rows of the Laver tables}\label{rows}
In this section I explain how to compute the row of $p$ for some particular values of $p$.
\subsection{The row of $2^n$}
It is easy to see, by reverse induction on $q$,  that $2^n*q=q\mod 2^n$ and  $\pi(2^n)=2^n$, moreover $\theta(2^n)=2^{n-1}$ since the period doubles between $2^{n-1}$ and $2^n$.
\subsection{The row of $p=2^m+2^n$}\label{2m2n}

\begin{proposition}\label{double}
Let $n>m\geq 0$  and  $p=2^m+2^n$ then $\pi(p)=2^{m+1}$ moreover if  $q<2^m$ then 
$p*q=q$ and $p*(q+2^m)=q+2^n$.

\end{proposition}
\begin{proof} 
 If  $2^n<r<2^m+2^n$ then $r=2^n+s$ with $s<2^m$ therefore $\pi(s)<2^m$ and $\pi(r)\leq 2^m$.
On the other hand, by (\ref{recurrence}), one has $$p*(2^l-k-1)=(p*(2^l-k))*(p-1)$$
therefore by induction on $k\leq 2^m$ one has  $p*(2^l-k)=p-k$ since
 $$(p*(2^l-k))*(p-1)=(p*(2^l-k))*(2^m-1)=(p*(2^l-k))-1.$$
 Finally  $p*(2^l-2^m)=2^n$
and $\pi(2^n+2^m)>\pi(2^m)$. The proposition follows from that. 
\end{proof}

\subsection{The row of $p=2^l+2^m+2^n$}
This case is more involved than the preceding ones.
Observe that $\pi(2^l+2^m)=2^{l+1}$ by Proposition \ref{double}, therefore $\pi(p)=2^{l+1}$ or $2^{l+2}$.
 We will  prove the following  result, by induction on $n$.

\begin{proposition}\label{row3}\
Let $n$ be a positive integer.
\begin{enumerate}
\item
For $0\leq l<m<n$ and $p=2^l+2^m+2^n$ one has
$\pi(p)=2^{l+2}$ if $l$ is even and $\pi(p)=2^{l+1}$ if $l$ is odd, moreover 
$\theta(p)=2^{l-1}$ in this last case. 
\item
 For  any $p<2^{n+1}$ one has $\pi(p)\leq 2^{n}$ with equality exactly in the following cases:
$$ 
\begin{aligned}
p&=2^{n}\\
p&=2^{n}+2^{n-1}\\
p&=2^{n}+2^{n-1}+2^{n-2}\quad\text{if $n$ is even.}
\end{aligned}
$$
\end{enumerate}
\end{proposition}

\textsc{Proof.}
The statements are  easy to check for small values of $n$. Let $r>0$ and   assume that Proposition \ref{row3} holds for all values  $l<m<n\leq r$. We will prove that it holds for all $l<m<n=r+1$.
The proof is divided into several cases.

\medskip

{\sl I. $r$ is odd}

We start with proving (1). If $m<r$ then the lines of $p=2^l+2^m+2^{r+1}$ and $p_-=2^l+2^m+2^{r}$ can be deduced from one another by   Proposition \ref{power1} therefore we conclude by  induction. Similarly, if 
$l<r-2$ and $m=r$ we can apply Proposition \ref{nmst} to conclude. 

It remains to consider the cases
$l=r-2$ and $l=r-1$. 

a) Let $l=r-2$ so that $p=2^{r+1}+2^{r}+2^{r-2}$.
Take some $t$ large enough (e.g. $t=r+2$) then one has $p-1=p*(2^t-1)$  and, by (\ref{recurrence}): 
\begin{equation}\label{rec1}
p*(2^t-k)=(p*(2^t-(k-1))*(p-1).
\end{equation}
Let $u$ be the smallest $s\geq 1$ such that $\pi(p-s)\geq 2^{r-1}$.
 Since $p-2^{r-3}=2^{r+1}+2^{r}+2^{r-3}$ one has, by what we just saw, that
$\pi(p-2^{r-3})=2^{r-1}$ therefore
 $u\leq 2^{r-3}$. Applying (\ref{rec1}) and noting that $p-1=2^{r-2}-1\mod 2^{r-2}$ one sees, by induction on $s$, that 
 $p*(2^t-s)=p-s$ for all $s\leq u$. 
Suppose that $u<2^{r-3}$ then $p-u=2^{r+1}+2^{r}+2^{r-2}-u=2^{r+1}+2^{r}+v$ with $2^{r-3}< v<2^{r-2}$. In order that 
$\pi(p-u)\geq 2^{r-1}$ one must have $\pi(v)\geq 2^{r-3}$ therefore, by induction hypothesis, either
 $v=2^{r-3}+2^{r-4}$ or $v=2^{r-3}+2^{r-4}+2^{r-5}$. However, by Proposition \ref{nmst}, one has $\pi(2^{r+1}+2^{r}+v)=
\pi(2^{r}+2^{r-1}+v)$ and one can use the induction hypothesis to see that in  these two cases 
$\pi(2^{r}+2^{r-1}+v)\leq 2^{r-2}$. It follows that $s=2^{r-3}$ and $p*(2^t-2^{r-3})=2^{r+1}+2^r+2^{r-3}:=w$ with $\pi(w)=2^{r-1}$ as we have seen. We can now compute: 
$$p*(2^l-2^{r-3}-1)=w*(p-1)=w*(2^{r-2}-1)=2^{r-3}+2^r-1<2^{r+1}.$$ It follows that 
$\theta(p)=2^{r-3}$ and $\pi(p)=2^{r-2}$, as claimed.

b) Let $l=r-1$. A similar reasoning as in case a) shows that $p*(2^l-k)=p-k$ for $k=1,\ldots,2^{r-2}$, in particular
$$p*(2^l-2^{r-2})=p-2^{r-2}=2^{r+1}+2^{r}+2^{r-2}.$$ We can now apply case a) and  compute 
$$p*(2^l-2^{r-2}-1)=(2^{r+1}+2^{r}+2^{r-2})*(p-1)=(2^{r+1}+2^{r}+2^{r-2}-1).$$ Using now that all $q$ with
$2^{r+1}+2^r<q\leq 2^{r+1}+2^{r}+2^{r-2}-1$ have period at most $2^{r-1}$ we obtain 
$p*(2^l-2^{r-1})=2^{r+1}+2^r$. We know that $\pi(2^{r+1}+2^r)=2^{r+1}$ and we can compute 
$(2^{r+1}+2^r)*(p-1)=2^{r+1}+2^{r-1}-1$. Since all $q<2^{r+1}+2^{r-1}$ have period at most $2^{r-1}$ it follows that
$p*(2^l-2^r)=2^{r+1}$ and $\pi(p)=2^{r+1}$.

Let us now prove (2). We have already checked that the periods of $p=2^{r+1}, 2^{r+1}+2^r$ and $2^{r+1}+2^r+2^{r-1}$ are equal to $2^{r+1}$. Suppose that some  other $p<2^{r+2}$ satisfies $\pi(p)=2^{r+1}$ then $p>2^{r+1}$ moreover if $p=2^{r+1}+q$ with $q<2^{r+1}$ then $\pi(q)= 2^r$ therefore by induction hypothesis one has either $q=2^r$ or $q=2^r+2^{r-1}$.

\medskip

{\sl II. $r$ is even}

The argument when  $r$ is even is similar.
We start with (1). The cases $m<r$ and $m=r,\,l<r-2$ are identical as the odd case.

Suppose now $p=2^{r+1}+2^{r}+2^{r-2}$.
Again we prove that $p*(2^l-k)=p-k$ for $k=1,\ldots,2^{r-3}$, in particular 
$p*(2^l-2^{r-3})=2^{r+1}+2^r+2^{r-3}$. This time one has $\pi(2^{r+1}+2^r+2^{r-3})=2^{r-2}$ and we can argue as in b) of case I to show that $\pi(p)=2^r$.

Assume that $p=2^{r+1}+2^{r}+2^{r-1}$, now   $p*(2^l-k)=p-k$ for $k=1,\ldots,2^{r-2}$, in particular 
$p*(2^l-2^{r-2})=2^{r+1}+2^r+2^{r-2}$ with $\pi(2^{r+1}+2^r+2^{r-2})=2^{r}$ as we have just seen and an argument similar as in case I above   shows that $\pi(p)=2^{r}$ and $\theta(p)=2^{r-2}$.

For the proof of (2), we have already checked that the periods of $2^{r+1}$ and of $2^{r+1}+2^r$  are $2^{r+1}$. Suppose that some  other $p<2^{r+1}$ satisfies $\pi(p)=2^{r+1}$ then $p>2^{r+1}$. Moreover if $p=2^{r+1}+q$ with $q<2^{r+1}$ then $\pi(q)\geq 2^r$ therefore by induction hypothesis one has either $q=2^r$ or $q=2^r+2^{r-1}$ or
$q=2^r+2^{r-1}+2^{r-2}$.  It remains to prove that $\pi(2^{r+1}+2^r+2^{r-1}+2^{r-2})=2^{r}<2^{r+1}$. The argument is similar as above: if one had $\pi(2^{r+1}+2^r+2^{r-1}+2^{r-2})=2^{r+1}$
one would have $(2^{r+1}+2^r+2^{r-1}+2^{r-2})*(2^l-2^{r-2})=2^{r+1}+2^r+2^{r-1}$ but by what we have seen above
$\pi(2^{r+1}+2^r+2^{r-1})=2^r$ therefore $\pi(2^{r+1}+2^r+2^{r-2})*(p-1)<2^{r+1}+2^r+2^{r-2}-1$.

\qed

\begin{remark} It would be also possible to prove the preceding proposition using  results of Dougherty \cite{Doug}
or Dr\'apal \cite{Dra4}. We provided the proof above in order to give a glimpse of the kind of arguments used in these computations.
\end{remark}

\section{Why are Laver table interesting?}\label{interest}
\subsection{Computation of the Laver tables}
In view of their very elementary definition and of their connection with basic operations such as the composition of functions, it is clear that Laver tables are fundamental objects in mathematics.
They are  also very recursive objects. This can be seen from  the relation (\ref{ld}) defining self-distributivity, in which the number of symbols on both sides is not the same. This implies that the computation of a product $p*q$ using rules (\ref{ld}) and (\ref{ld1}) involves going through long recursions.  As I wrote in the introduction, in some respect these objects are reminiscent of  the Mandelbrot or Julia sets which  also have a very simple recursive definition, yet display very complex features\footnote{Incidentally the combinatorics of powers of $2$ also plays a role in the study of the Mandelbrot set, cf \cite{Milnor1}.}.
As such, they pose challenging problems which are quite different from the ones one encounters usually in algebraic combinatorics. The most obvious question  is whether there exists a simple formula for computing $p*q$ for arbitrary $p$ and $q$, i.e. one  which involves computing  as few intermediate values as possible. Also one would like to find formulas for $\pi(p)$ or $\theta(p)$. As we shall see below, it is unlikely that such  general formulas exists, however one might find large classes of elements for which the computation is easy. Some  examples  were described above in section \ref{rows} and I will describe further results in this direction in section \ref{Max}. Examination of numerical values of the Laver tables reveals a lot of non obvious structure. As an example, 
 here are  the periods $\pi(p)$ of all $p\in[1,256]$, with the $p$ arranged in increasing order,  in  16 rows of 16. 

\begin{center}
\scalebox{0.8}{
$\begin{array}{cccccccccccccccc}
1&2&2&4&2&4&4&8&2&4&4&8&4&4&4&16\\
2&4&4&8&4&4&4&16&4&4&4&16&8&8&8&32\\
2&4&4&8&4&4&4&16&4&4&4&16&8&8&8&32\\
4&4&4&16&4&4&4&16&4&4&4&16&16&8&8&64\\
2&4&4&8&4&4&4&16&4&4&4&16&8&8&8&32\\
4&4&4&16&4&4&4&16&4&4&4&16&16&8&8&64\\
4&4&4&16&4&4&4&16&4&4&4&16&16&8&8&64\\
8&4&4&16&4&4&4&16&4&4&4&16&16&16&8&128\\
2&4&4&8&4&4&4&16&4&4&4&16&8&8&8&32\\
4&4&4&16&4&4&4&16&4&4&4&16&16&8&8&64\\
4&4&4&16&4&4&4&16&4&4&4&16&16&8&8&64\\
8&4&4&16&4&4&4&16&4&4&4&16&16&16&8&128\\
4&4&4&16&4&4&4&16&4&4&4&16&16&8&8&64\\
4&4&4&16&4&4&4&16&4&4&4&16&16&8&8&64\\
4&4&4&16&4&4&4&16&4&4&4&16&16&8&8&64\\
16&4&4&16&4&4&4&16&4&4&4&16&16&16&8&256
\end{array}$}

\end{center}

A cursory look at this table suggests that  there are many patterns here. For example several lines coincide and some of these coincidences can be easily explained using Propositions \ref{power1} or \ref{power2}. More subtle is the fact that  the first row is equal to the first column and multiplying it  by $16$ gives the last column.  This remark can be explained by the following results.
\begin{proposition}{Dr\'apal \cite{Dra1}, Theorem 3.6.}\label{draphom}

Let $\sigma_d:([1,2^n],\star_n)\to ([1,2^{n+d}],\star_{n+d})$ be given by $\sigma_d(p)=2^dp$ then $\sigma_d$ is a homomorphism if and only $n\leq 2^{r+1}$ where $2^r$ is the largest power of $2$ dividing $d$.
Similarly $\sigma_d$ is a homomorphism for $*$.
\end{proposition}
Applying Proposition \ref{draphom}  for $d=4$ and $n=4$ on sees that the map $p\mapsto 16p$ is a homomorphism from $[1,16]$ to $[1,256]$ therefore the period of $256-16q$ in the Laver table of size $256$ is $16$ times the period of $16-q$ in the Laver table of size $16$. This gives an explanation for the fact that the first row multiplied by 16 is equal to the last column. A corollary of Proposition \ref{draphom} is the following.
\begin{cor}\label{cordraphom}
For any integers $r\in[0,2^{2^n}[$ and $q\geq 0$ one has
$$2^{2^n}((1+r)*q)=(1+r2^{2^n})*q.$$
\end{cor}
\begin{proof}
Applying  (\ref{puissancesback}) and Proposition \ref{draphom} one gets

$(1+r2^{2^n})*q=(r2^{2^n})^{(q)}=2^{2^n}\left(r^{(q)}\right)=2^{2^n}((1+r)*q).$
\end{proof}
It follows from this Corollary  that the first row and the first column of the table above coincide.
These results can be further generalized and several more sophisticated explicit homomorphisms have been constructed by
Dr\'apal (see e.g. \cite{Dra2}, \cite{Dra4}) and Dougherty (see \cite{Doug}, we will use one of these results in section \ref{Max}).
Looking for more homomorphisms could be of use for solving the questions we are going to expose in the next section.
\subsection{Asymptotic properties of Laver tables}
For any $p\geq 1$, the sequence of periods 
$\pi_n(p)$ (which is defined for $n$ large enough) satisfies $\pi_{n+1}(p)=\pi_n(p)$ or $2\pi_n(p)$, in particular it is nondecreasing.  It is therefore natural to ask whether it remains bounded as $n$ goes to infinity. By work of 
Laver \cite{La}, \cite{La2},  Dougherty  and Jech \cite{DougJech} one knows that this  is equivalent to  the freeness of   the projective limit of Laver tables, more precisely, $\pi_n(1)\to\infty$ if and only if $\pi_n(p)\to\infty$ for all $p$, if and only if  the projective limit of the Laver tables is the free left distributive system generated by $1$. Moreover, if one assumes the existence of a limit ordinal $\lambda$ and a nontrivial elementary embedding 
 from $V_\lambda$ into itself, then all these equivalent statements hold.  No direct proof of this (i.e. not using the Laver axiom) has been found and  by \cite{DougJech}  such a proof does not exist in primitive recursive arithmetic. In fact the function which maps $m$ to the smallest $n$ such that $\pi_n(1)=2^m$ grows faster than any primitive recursive function.  For example, it is easy to compute the first values of $\pi_n(1)$ and see that $\pi_9(1)=16$, however Dougherty \cite{Dou} has given an amazing lower bound:
he proved that, if there exists an $n$ such that $\pi_n(1)=32$ then $n>f_9(f_8(f_8(254)))$ where $f_x(y)$ is a variant of the Ackerman function. This number, if it exists, is thus incredibly large. 
However it is not known whether the existence of such an $n$ can be proved without Laver's axiom.
Finding proofs of these statements from a combinatorial approach is a very challenging task. Some interesting attempts have been made by Dr\'apal \cite{Dra2} and Dougherty \cite{Doug} but for the moment a full proof seems out of reach.
The considerations above give an  answer to a question we raised in section \ref{LD1}. There we mentioned the problem of constructing a
family $I$ (not reduced to the identity) of injections  of a set $X$ into itself, endowed with a binary left distributive operation $\star:I\times I\to I$, such that, for any $\iota,\eta\in I$ and $x\in X$ one has
$\iota\star\eta(\iota(x))=\iota(\eta(x))$, as in (\ref{function}). As shown by Dougherty and Jech \cite{DougJech}, assuming that the projective limit of Laver tables is free, one can  construct in the following way such a family. Let $W$ be the free left distributive system with one generator. For any $w\in W$ one can look at its image in the Laver table of order $2^m$. Let $e_w(n)$ be the largest integer $m$ such that $w\star_m2^n=2^m$ in the Laver table of order $2^m$. Since the period $\pi_m(2^n)$ goes to infinity as $m\to\infty$ (by Laver's result) this is well defined and gives a family $e_w;w\in W$ of embeddings of $\bf N$ into itself. This family is endowed with a left-distributive operation inherited from $W$. Dougherty and Jech have proved that for any pair of such embeddings the property
$\iota\star\eta(\iota(n))=\iota(\eta(n))$ is satisfied. Actually they prove much more properties of these objects, see \cite{DougJech}.

Another natural question that one can ask is whether one can say something about Laver tables seen from very far, i.e. do the Laver tables satisfy some interesting statistical properties for large $n$? A first result in this direction is the subject of the next section.

\section{Asymptotic frequencies}\label{freq}
For $n\geq k$ let $N_k(n)$ be the number of $p\in[1,2^{n}]$ whose period is $2^k$ and $\omega_k(n)=N_k(n)/2^n$ be the frequency of the period $2^k$ in the Laver table of order $2^n$. Thus $\sum_{k=0}^n\omega_k(n)=1$. 

\begin{proposition} For any $k$ the limit $\omega_k=\lim_{n\to\infty}\omega_k(n)$ exists.
\end{proposition}
\begin{proof}
Let $P_k(n)$ be the number of $p\in[1,2^{n-1}]$ such that $\pi(p)=2^{k-1}$ and $\pi(p+2^{n-1})=2^k$. Since $\pi(p+2^{n-1})=\pi(p)$ or $2\pi(p)$ one has 
$$N_k(n)=2N_k(n-1)+P_{k}(n)-P_{k+1}(n)$$ and $$\omega_k(n)=\omega_k(n-1)+(P_{k}(n)-P_{k+1}(n))/2^n.$$
The only $p$ with $\pi(p)=1$ is $p=1$ and the only $p$ with $\pi(p)=2$ are the numbers
$p=1+2^k;k=0,1,2,\ldots$. It follows that $N_0(n)=1$ and $N_1(n)=n$.
Looking at the table of order 4 we see that  $\omega_2(2)=1/4$ , moreover from Proposition \ref{double} we see that
 $P_2(n)=N_1(n-1)=n-1$ for $n\geq 3$ therefore
$$
\begin{array}{rcl}
\omega_0(n)&=&\frac{1}{2^n}\to \omega_0=0\quad\text{as}\ n\to\infty\\
\omega_1(n)&=&\frac{n-1}{2^n}\to \omega_1=0\quad\text{as}\ n\to\infty\\
\omega_2(n)&=&1/4+\sum_{m=3}^n\frac{m-1}{2^m}-\frac{P_3(m)}{2^m}.
\end{array}
$$
Since $\omega_2(n)\geq 0$ one has 
$$\sum_{m=2}^n\frac{P_3(m)}{2^m}\leq 1/4+\sum_{m=3}^n\frac{m-1}{2^m}<1/4+\sum_{m=3}^\infty\frac{m-1}{2^m}=1$$
therefore the series $\sum\frac{P_3(m)}{2^m}$ converges and 
$$\omega_2(n)\to\omega_2=1-\sum_{m=3}^\infty\frac{P_3(m)}{2^m}\quad \text{as}\ n\to\infty$$
One has also
$$\omega_3(n)=1/8+\sum_{m=4}^n\frac{P_3(m)}{2^m}-\frac{P_4(m)}{2^m}$$
therefore 
 $$\sum_{m=4}^n\frac{P_4(m)}{2^m}\leq 1/8+\sum_{m=4}^\infty\frac{P_3(m)}{2^m}$$ and the series 
$\sum_{m=4}^\infty\frac{P_4(m)}{2^m}$ converges. It follows that

$$\omega_3(n)\to 1/8+\sum_{m=4}^\infty\frac{P_3(m)}{2^m}-\frac{P_4(m)}{2^m}\quad\text{as}\ n\to\infty$$
An obvious induction now yields the result for all $k$ as well as the explicit expressions
$$\omega_k= 2^{-k}+\sum_{m=k+1}^\infty\frac{P_k(m)}{2^m}-\frac{P_{k+1}(m)}{2^m}$$
where all series $\sum_m\frac{P_k(m)}{2^m}$ converge.
\end{proof}

Note that, by  the same argument, one can also  obtain the convergence of frequencies $\omega_{k,l,m}$ of the set of $p\leq 2^n$, with period $2^k$ satisfying,
$p=m\mod 2^l$ for some fixed $l$ and $m\in[1,2^l]$.
The argument above is very simple, moreover one can also derive from it upper bounds on the values of the asymptotic frequencies, for example one has 
$$\sum_{k=1}^l\omega_k=1-\sum_{m=l+1}^\infty\frac{P_{l+1}(m)}{2^m}\leq \sum_{k=1}^l\omega_k(n)$$so that knowing the values of the $\omega_k(n)$ for some $n$ gives upper bounds on the asymptotic frequencies.
Unfortunately this proof does not give any useful lower bounds on the asymptotic frequencies. Indeed, at this point, we cannot exclude that all $\omega_k$ may be zero.  Here are some numerical values of the frequencies (expressed as percentages to gain readability) for $n=22$ and $n=31$:

\begin{center}
\begin{tabular}{|c|c|c|c|c|c|c|}
\hline
$\pi_n$&2&4&8&16&32
\\
\hline
$n=22$&0.000572 &52.936697 &10.196733 &30.197978 & 0.002623 
\\
\hline
$n=31$&0.000002 & 52.936599 & 10.193012 & 30.202195 & 0.001429 
\\
\hline
$\pi_n$&64&128&256&512&1024
\\
\hline
$n=22$ & 3.550982 & 0.684047 & 2.035284 & 0.000763&  0.209165
\\
\hline
$n=31$ & 3.551050 & 0.679749 & 2.040756 & 0.000003& 0.209228
\\
\hline
\end{tabular}
\end{center}

Arguments of Dougherty in \cite{Dou} show that one can expect the frequencies $\pi_n(p)$ to grow very slowly with $n$,  therefore that the numbers $P_k(n)$ might be small compared to $2^n$.
This would imply that not only some of the $\omega_k$ are nonzero, but also that 
$\sum_n\omega_n=1$, therefore they define a probability distribution on the positive integers.
 Here are some values of ${\mathcal P }(n)=\sum_k P_n(k)$ which is the number of $p$ whose period doubles between $A_n$ and $A_{n+1}$.

\begin{center}
\begin{tabular}{|c|c|c|c|c|c|c|c|}
\hline
$n$&4&8&12&16&20&24&28
\\
\hline
${\mathcal P}(n)$&16 & 58 &147 &336 & 650&1201&2249
\\
\hline
$2^{4+n/4}$&32 & 64 &128 &256 &512 & 1024& 2048
\\
\hline
\end{tabular}
\end{center}
On this small sample the formula  $2^{n/4+4}$ seems to give a rough approximation of ${\mathcal P}(n)$ which, if it remains valid for large values of $n$,  would be enough for the probability measure to put zero mass at infinity.

The existence of this probability measure, which is canonical, in the sense that it does not depend on any external parameter like a mean or a variance, raises challenging and completely open questions:
Can one characterize this probability distribution, or give a formula for it?  Does it appear in other questions of mathematics or physics?

\section{Maximal elements}\label{Max}
\subsection{Definition of maximal elements}
Let $bit(n)$ is the number of ones in the binary expansion of an integer $n$. The set of numbers $r$ satisfying
$r\sqsubset n$ has $2^{bit(n)}$ elements.
We have seen that for any $p$ the numbers $p*q$ satisfy 
$p*q\sqsubset p-1$ therefore
\begin{equation}\label{bit}
\pi(p)\leq 2^{bit(p-1)}.
\end{equation}
\begin{definition}
A positive integer $p$ is called {\sl maximal} if $\pi(p)= 2^{bit(p-1)}$.
\end{definition}
 It follows from section \ref{rows} that, if $p$ has at most two ones in its binary expansion, then $p$ is maximal. If it has three ones in its binary expansion then it is maximal if and only if the largest power of $2$ which divides it is of the form $2^{2m}$ for some integer $m\geq 1$.
The purpose of this section is to determine the set of all maximal $p$'s. As a consequence of the description of these elements, we will see that they form a subset  stable under the binary operations $*$ and $\circ$.
Let $p$ be a maximal element, it follows from the definition that the numbers $p*q,q=0,1,2\ldots \pi(p)-1$ are all the integers whose binary expansion is contained in that of $p-1$, listed in increasing order. This gives a simple algorithm for the  computation of $p*q$. As an example,  assume $p$   has binary expansion $p=1010110000111100000001$ (we will see below that $p$ is maximal, thus $\pi(p)=2^8$) and let
 $q$ have binary expansion $11000101$.  The binary expansion of 
$p*q$ is obtained by 
writing the binary expansion of $q$ below the $1$'s of $p-1$ then keeping only the $1$'s of $p-1$ which match a $1$ of $q$, as shown below:

\medskip

$$
\begin{array}{ccccccccccccccccccccccccccccccccccccc}
p-1&1&0&1&0&1&1&0&0&0&0&1&1&1&1&0&0&0&0&0&0&0&0\\
q& 1&&1&&0&0&&&&&   0&1&0 & 1&\\
p*q&1&0&1&0&0&0&0&0&0&0&0&1&0&1&0&0&0&0&0&0&0&0\\
\end{array}
$$

 Note also that, by Proposition \ref{power1}, if $p$ is maximal then all  the numbers $p\mod(2^m)$ are maximal, moreover if we write
$p=2^n+q$ with $q<2^n$ then for any $l$ such that $2^l>q$ the number $2^l+q$ is also maximal. It follows that it is enough to describe all maximal elements such that $2^{n-1}+2^{n-2}< p\leq 2^n$ for some $n$. 

\subsection{Characterization of maximal elements}
Recall that a partition of an integer $n$ is a sequence of integers $\lambda_1\geq \lambda_2\geq \ldots\geq 0$
such that $n=\sum_i\lambda_i$. A binary partition is a partition in which all $\lambda_i$ are powers of $2$.
Since any binary partition of an odd integer must contain a $1$, the number of binary partitions of $2n+1$ is equal to the number of binary partitions of $2n$. The numbers of binary partitions of $n$ form sequence A018819 in OEIS, see \cite{OEIS}. Their first values are $1, 1, 2, 2, 4, 4, 6, 6, 10, 10, 14, 14, 20, 20,  \ldots$. These numbers satisfy the recursion
$$a(2m+1)=a(2m)=a(2m-1)+a(m).$$

\begin{theorem} For any $n\geq 2$ the maximal elements $p\in]2^{n-1}+2^{n-2},2^n]$ are in bijection with the binary partitions with sum  $n-1$: if 
$$n-1=(b_1+1)2^{a_1}+(b_2+1)2^{a_2}+\ldots+(b_r+1)2^{a_r},$$ with $a_1<a_2<a_3<\ldots<a_r$ and $b_i\geq 0$, is a binary partition of sum $n-1$ (where $b_i+1$ is the multiplicity of $2^{a_i}$ in the partition) then the binary word
\begin{equation}\label{maximal}
11^{2^{a_1}}0^{b_12^{a_1}}1^{2^{a_2}}0^{b_22^{a_2}}\ldots 1^{2^{a_r}}0^{b_r2^{a_r}}
\end{equation} 
is the binary expansion of $p-1$, where $p$ is a maximal element of period $2^m$ with $m=1+2^{a_1}+\ldots +2^{a_r}$, and all maximal elements in $]2^{n-2}+2^{n-1},2^n]$ are of this form.
\end{theorem}
\begin{cor}The set of maximal elements is the set of $p$ such that $p-1$ has binary expansion of the form
\begin{equation}\label{maximal2}
10^{b_0}1^{2^{a_1}}0^{b_12^{a_1}}1^{2^{a_2}}0^{b_22^{a_2}}\ldots 1^{2^{a_r}}0^{b_r2^{a_r}}.
\end{equation}
\end{cor}

The proof relies on the following results of Dougherty (see Theorem 4.1 and Lemma 4.5 in \cite{Doug}).
\begin{theorem}\label{Doughom}\

\begin{enumerate}[label=(\roman*)]
\item
If $p-1=a2^{2^k}+b$ with $a<2^{2^k}$ and $b<2^{2^k}-1$ then $\pi(p)\leq 2^{2^k}$.
\item
Let $p=x+1+2^{mn}y$ where $n=2^k$ is a power of $2$ and $x=z2^{(m+1)n}$ for some positive integer $z$ and $y<2^n$. Let $2^l$ be the period of $x+1$ and suppose $l\leq n$, then $p$ has the same period as  
$2^{l+n}-2^n+y+1$ moreover $(2^{l+n}-2^n+y+1)*q=2^nw+y'$ with $y'<2^n$ if and only if
$p*q=(x+1)*w+2^{mn}y'$.
\end{enumerate}
\end{theorem}

From this we deduce
\begin{lemma}\label{maximal3}
Let $k\geq 1,m\geq 0$ and $p-1=(2^{2^k+1}-1)2^{2^{k-1}+m2^k}$ then $\pi(p)\leq 2^{2^k}$ 
\end{lemma}
\begin{proof}
If $m=0$ apply  $(i)$ of  Theorem \ref{Doughom} with $a=2^{2^{k-1}+1}-1$ and $b=2^{2^{k}}-2^{2^{k-1}}$ to get that 
$\pi(p)\leq2^{2^k}$. If $m>0$ one can use $(ii)$ with $x+1=(2^{2^{k-1}+1}-1)2^{(m+1)2^k}$ and
$y=2^{2^k}-2^{2^{k-1}}$ so that $p=x+y2^{m2^k}+1$. Since $bit(x)=2^{k-1}+1$ one has $\pi(x+1)=2^l\leq 2^{ 2^{k-1}+1}$ therefore $p$ has the same period as 
$2^{l+2^k}-2^{2^k}+y+1=2^{l+2^k}-2^{2^{k-1}}+1$, which is less that $2^{2^k}$ by the case $m=0$ that we just proved.
\end{proof}
In particular $p=(2^{2^k+1}-1)2^{2^{k-1}}+1$ is not a maximal element.

\textsc{Proof of Theorem \ref{maximal}.}
We prove by induction on $m=1+\sum_i 2^{a_i}$ that, if  $p-1$ has binary expansion (\ref{maximal}) then it is a maximal element.
This is clear if $m=1$ or $2$, by Proposition \ref{row3}.
If $m-1=2^a+l$ with $l<2^a$ then
$p-1$ has binary expansion $1w1^{2^a}0^{b2^a}$ for some binary word $w$, moreover if   $x$ has binary expansion
$1w0^{(b+1)2^a}$ then, by induction hypothesis, $x+1$ is a maximal element.
We can now apply $(ii)$ of  Theorem \ref{Doughom} with $y=2^{2^a}-1,\,n=2^a,\,m=b$ to see that $p$ is again maximal.

Now we prove the converse: if $p-1$ is not of the form (\ref{maximal2}) then $p$ is not maximal. The reduction modulo $2^m$ of a maximal element is again maximal, therefore
we may assume that $p-1=2^n+q-1$ where $q-1<2^n$ and $q$ has the form (\ref{maximal2}), i.e. $n$ is the smallest integer such that $p-1\mod 2^{n+1}$ does not have the form  (\ref{maximal2}). It is easy to check that 
 $p-1$ must have a binary expansion of the form
$$10^{b_0}1^{2^{a_1}}0^{b_12^{a_1-1}}1^{2^{a_2}}0^{b_22^{a_2}}\ldots 1^{2^{a_r}}0^{b_r2^{a_r}}$$
where 
\begin{enumerate}
\item
either $r\geq 1,\ 1\leq a_1\leq a_2< a_3<\ldots<a_r$ and  $b_1$ is odd.
\item or
$r\geq 2,\
 0\leq a_1= a_2< a_3<\ldots<a_r$ and  $b_1\geq 2$ is even.

\end{enumerate}
Without loss of generality we can assume that $b_0=0$ (cf Prop. \ref{power1}).

We first treat case $(1)$ for $r=1,b_1=1$: this follows from Lemma \ref{maximal3} with $a_1=k$. 
 Extension to  $b_1>1$ is obtained by applying  Theorem \ref{Doughom}, and then $r\geq 2$ follows by induction, again using Theorem \ref{Doughom}.

Case $(2)$ with $r=2$: let $a=a_1$ and  assume by contradiction that $p$ is maximal, with period $2^{2^{a+1}+1}$, then  
$$p*(2^{2^{a+1}+1}-2^{2^a}+1)-1=p'-1$$  can be computed using the algorithm above for maximal elements. One sees that $p'-1$
 has binary expansion    $11^{2^a}0^{(b_1+1)2^a}1^{b_22^a}$ therefore $p'$ is maximal by what we have proved. One can compute again
 $$p*(2^{2^{a+1}+1}-2^{2^a})=p'*(p-1)=2^{(b_1+b_2+1)2^a}-2^{(b_1+b_2)2^a}<2^{(b_1+b_2+1)2^a+1}$$ which shows that $p$ is not maximal and also that $\theta(p)\leq 2^{2^a}-1$.
 The case of $r>2$ is treated again by induction.

\qed

\subsection{Stability of maximal elements}
Let $\mathcal M$ be the set of maximal elements and $\mathcal M^*=\mathcal M\cup\{0\}$.
\begin{theorem} Let $p,q\in \mathcal M^*$  then $p*q\in\mathcal M^*$,
if  $p,q\in \mathcal M$  then $p\circ q\in\mathcal M$.
\end{theorem}
\begin{proof}
Now that we have a complete description of maximal elements and a simple formula for computing $p*q$ and $p\circ q=p*(q-1)+1$ whenever $p$ is maximal, this amounts just to a verification. This is not difficult but a bit cumbersome due to the boundary effects caused by substracting $1$, so we will only sketch the idea of the proof, leaving the details to the reader.
First we note that if $p$ maximal,  then the binary expansion of $p-1$ can be obtained from the following construction.
Let $w$ be a binary word and $t=0^{b2^a}$. Define the insertion of $t$ into $w$ as follows: split $w$ as $w=uv$ where the number of $1$'s in $u$ is at most $2^{a+1}$ while the number of $1$'s in $v$ is a multiple (possibly $0$) of $2^{a+1}$ then insert $t$ so as to obtain the word $utv$. The  result of the insertion is uniquely defined by $w,a,b$, 
but beware that by writing $b2^a=c2^d$ the insertion may give a different result. It is easy to see that  the set of maximal elements with period $2^n$ is exactly the set of $p$ such that the binary expansion of $p-1$ is obtained by a sequence of such insertions, starting  from the word  $1^n$.

Let now $p$  be a maximal element with period $2^n$ and let $x_0=2^n-1,\ldots,x_s=p-1$ be the sequence of numbers  obtained by the successive insertions. Then each $p_i=x_i+1$ is maximal and for any $q$, the binary expansion of  $p_{i+1}*q-1$ is obtained from that of $p_i*x$ by an insertion of the same block of zeros as $x_{i+1}$ from $x_i$. Note that $b2^a$ may be interpreted as $c2^d$ for some $d<a$ in this process.
Since $q$ is maximal,  $2^n*q$ is maximal and
since the insertion process preserves the set of maximal elements we are done.
The case of $\circ$ is similar and left to the reader.
\end{proof}
We have thus identified a subset of the integers  which forms a stable subset for the two operations $*,\circ$, on which these are given by very simple formulas. This could be the first step in determining a more general formula, valid for much larger sets of integers.

%$$\includegraphics[scale=.6]{laver_frequences.png}$$

\end{document}